\documentclass[reqno]{
amsart}

\usepackage{hyperref}

 \hypersetup{colorlinks=true, urlcolor=green, citecolor=red, linkcolor=blue}

\usepackage{latexsym,amssymb, amsthm,amsfonts,amsmath,amssymb,amstext,amsthm}
\usepackage[mathscr]{euscript}
\usepackage[abbrev]{amsrefs}
\usepackage[all]{xy}
\usepackage{graphicx}
\usepackage{tikz}
\usepackage{pdfpages}
\usepackage{xcolor}
\usepackage{tabularray}
\usepackage{bbm,dsfont}

\newcommand{\bb}{\mathbb}
\newcommand{\cb}{\overline{C}} 
\newcommand{\ccc}{\mathscr{C}}
\newcommand{\cliff}{{\rm Cliff}} 
\newcommand{\cp}{\ccc_{P}} 
\newcommand{\ff}{{\rm F}}
\newcommand{\fff}{\mathscr{F}} 

\newcommand{\gap}{{\rm G}}
\newcommand{\gon}{{\rm gon}} 
\newcommand{\im}{{\rm im}}
\newcommand{\kk}{{{\rm K}}}
\newcommand{\mc}{\mathcal}
\newcommand{\nn}{\mathbb{N}}
\newcommand{\obp}{\overline{\mathcal{O}}_P}

\newcommand{\oo}{\mathcal{O}} 
\newcommand{\op}{\mathcal{O}_P} 
\newcommand{\pb}{\overline{P}} 
\newcommand{\sss}{{\rm S}}

\newcommand{\ww}{\omega} 
\newcommand{\wwp}{\omega_{P}} 

\def\sbsno{(\arabic{section}.\arabic{subsection})\enspace}
\def\defspec#1{\def\headspec{#1}%
  \ifx\headspec\empty 
  \else{\unkern\enspace(#1)}\fi
}       

\newtheoremstyle{italics}
  {6pt}
  {6pt}
  {\itshape}
  {}
  {\bfseries}
  {.}
  {.5em}
  {\sbsno \thmname{#1}{\rm \defspec{#3}}}

\newtheoremstyle{roman}
  {6pt}
  {6pt}
  {\rmfamily}
  {}
  {\bfseries}
  {.}
  {.5em}
  {\sbsno \thmname{#1}\defspec{#3}}

\theoremstyle{italics}
 \newtheorem{lem}[subsection]{Lemma}

 \newtheorem{prop}[subsection]{Proposition}
 \newtheorem{thm}[subsection]{Theorem}

\theoremstyle{roman}
 \newtheorem{defi}[subsection]{Definition}

 \newtheorem{rem}[subsection]{Remark}
 
 \newtheorem{sbs}[subsection]{} 

\numberwithin{equation}{subsection}

\numberwithin{equation}{subsection}

\DefineSimpleKey{bib}{primaryclass}{}
\DefineSimpleKey{bib}{archiveprefix}{}

\BibSpec{arXiv}{%
  +{}{\PrintAuthors}{author}
  +{,}{ \textit}{title}
  +{}{ \parenthesize}{date}
  +{,}{ arXiv }{eprint}
  +{,}{ primary class }{primaryclass}
}

\begin{document}

\title{On Enriques-Babbage Theorem for singular curves }

\thanks{}



\author[L. Feital]{Lia Feital}
\address{Departamento de Matem\'atica, CCE, UFV
Av. P H Rolfs s/n, 36570-000 Vi\c{c}osa MG, Brazil}
\email{liafeital@ufv.br}

\author[N. Galdino]{Naam\~a Galdino}
\address{IMECC, Unicamp,
Rua S\'ergio Buarque de Holanda, 651, 13083-859 Campinas SP, Brazil}
\email{naama@ime.unicamp.br}

\author[R. V. Martins]{Renato Vidal Martins} 
\address{Departamento de Matem\'atica, ICEx, UFMG
Av. Ant\^onio Carlos 6627,
30123-970 Belo Horizonte MG, Brazil.}
\email{rvidalmartins@gmail.com}

\author[A. Souza]{\'Atila Souza}
\address{Departamento de Matem\'atica, ICEx, UFMG
Av. Ant\^onio Carlos 6627,
30123-970 Belo Horizonte MG, Brazil.}
\email{desouzaatilafelipe@gmail.com}





\begin{abstract}
We propose a version of the Enriques-Babagge Theorem for a singular curve $C$, involving its canonical model $C'$. We provide a partial proof for an arbitrary curve $C$ and complete the proof for unicuspidal monomial curves by describing the generators of the ideal of $C'\subset\mathbb{P}^{g-1}$.
\end{abstract}

\maketitle


\section{Introduction}\label{Intro}

Let $C$ be an integral and complete curve over an algebraically closed field of arbitrary characteristic. First, suppose $C$ is smooth. In late 19th century, Max Noether proved in \cite{N} a celebrated result establishing that every canonical curve is projectively normal.  

Noether's result formed the basis for subsequent works by Enriques \cite{En}, Babbage \cite{B}, and Petri \cite{P}. From these articles one derives the following result: if $C$ is a canonical curve, then one, and only one, of the following three statements holds: (i) $C$ is cut out by quadrics; (ii) $C$ is trigonal; and (iii) $C$ is isomorphic to a plane quintic; later known as the Enriques-Babbage Theorem \cite{ACGH}*{p.\,124}.

In 1984, Green proposed in \cite{Gr}*{Conj.\ (5.1)} a conjecture that generalizes both results above. He called it \emph{Noether-Enriques-Petri Conjecture}, nowadays known as \emph{Green's Conjecture}:
$K_{p,2}(C,\omega)=0$ if and only if $p < \cliff(C)$
for all $p \geq 0$, relating the Koszul cohomology \cite{Gr}*{p.\,126} of the dualizing sheaf $\ww$ of $C$ to the Clifford index \cite{Martens}*{p.\,83} of $C$. A thorough study of this problem can be found, for example, in \cite{ApF}. The conjecture was proved for a general smooth curve by C. Voisin in \cites{Vo1, Vo2}. We give a more detailed description of the conjecture at the beginning of Section \ref{kcgc} and adjusting it to singular curves, allowing torsion free sheaves in the definitions of both the Koszul Cohomology and the Clifford index.

The connection with the above discussion is clear. If $p=0$, then the conjecture is equivalent to Max Noether's Theorem. Indeed, note that $K_{0,2}(C,\omega)=0$ is equivalent to saying that 
${\rm Sym}^2  H^0(C,\ww)\twoheadrightarrow H^0(C,\ww^2)$, which is sufficient to obtain the surjections for higher orders. On the other hand, $\cliff(C)>0$ is equivalent to saying that $C$ is non-hyperelliptic and, therefore, isomorphic to a canonical curve.

If $p=1$, then the conjecture is equivalent to the Enriques-Babbage Theorem. Indeed, suppose that $C$ is non-hyperelliptic. Then, $K_{1,2}(C,\ww)=0$ if and only if $C$, viewed as a canonical curve, is cut out by quadrics. On the other hand, $\cliff(C)=1$ if and only if $C$ is trigonal or isomorphic to a plane quintic.

Now assume $C$ is Gorenstein, possibly singular. Rosa-St\"ohr in \cite{RSt}*{Thm.\ 1} combined with St\"ohr-Viana in \cite{SV2}*{Thm.\ 1.1} characterized trigonal curves as follows: let $C$ be a non-hyperelliptic curve of arithmetic genus $g\geq 5$; then $C$ is trigonal if and only if $C$, viewed as a canonical curve, lies on a $2$-dimensional rational normal scroll in $\mathbb{P}^{g-1}$; if the scroll is a cone, then $C$ is singular and passes through the vertex (see \cite{KM2}*{Sec.\,4} for more details on scrolls). 

The aim of this work is extending the Enriques-Babbage statement for a larger class of curves, possibly non-Gorenstein. To this end, we introduce some concepts which we describe next. The key-ingredient to get results matching classical ones is the following definition: let $\pi :\cb\rightarrow C$ be the normalization map; the linear system $(\pi^*\ww/{\rm Torsion}(\pi^*\ww),H^0(\ww))$
defines a morphism $\psi :\cb\rightarrow{\mathbb{P}}^{g-1}$; call $C':=\psi(\cb)$ the \emph{(Rosenlicht) canonical model} of $C$ \cite{R}*{p.\,188\,top} (see \eqref{cm}). Now given $P\in C$, let $\op$ be the local ring, $\cp$ the conductor, and  $\delta_{P} := \dim (\overline{\mathcal{O}}_{P}/\mathcal{O}_{P})$ be the singularity degree. Set $\eta_{P}:=\delta_{P}-\dim (\mathcal{O}_{P}/\cp)$. Following \cite{BF}*{pp. 418, 433, Prps. 21, 28}, we say $C$ is \emph{Kunz} if it has only one non-Gorenstein point $P$, and such that $\eta_P=1$. Finally, recall that an embedded projective variety is said \emph{linearly normal} if the linear system of hyperplanes is complete. With this in mind, we have the following. 


\medskip

\noindent {\bf Theorem 1.} \emph{Let $C$ be an integral and complete curve over an algebraically closed field. Assume the canonical model $C'$ is linearly normal. The following hold:
\begin{itemize}
\item[(I)] Assume $C$ is non-Gorenstein. Then
\begin{itemize}
    \item[(i)] $K_{1,2}(C,\ww) = 0$ if and only if $C'$ is cut out by quadrics; 
    \item[(ii)] $C'$ is cut out by quadrics and cubics; 
    \item[(iii)] if $C$ is not Kunz, then $C'$ is cut out by quadrics. 
\end{itemize}
\item[(II)] If $C$ is unicuspidal and monomial, then one, and only one, of the following statements is valid:
\begin{itemize}
    \item[(i)] $C'$ is cut out by quadrics; or
    \item[(ii)] $C$ is trigonal and Gorenstein; or
    \item[(iii)] $C$ is isomorphic to a plane quintic; or 
    \item[(iv)] $C$ is Kunz,
\end{itemize}
Moreover, if $C$ is non-Gorenstein and trigonal, then if the $g_{3}^{1}$ is base point free then $C$ satisfies (iv); otherwise, it may be included in cases (i) or (iv)
\end{itemize}}

The above result is proved in \eqref{Kleiman} and \eqref{gororngor}. In itens (I).(ii) and (iii), we are essentially rephrasing the last assertion in \cite{KM}*{Thm.~6.4}. We discuss, in \eqref{remebb}, the possibility of dropping $C$ being unicuspidal monomial in (II). In fact, to get (II) in general we only need: (a) a version of Enriques-Babbage Theorem for Gorenstein curves (which could be obtained via, for instance, \cites{CS,OS,Sc,St1}) and (b) showing that a Kunz curve cannot be cut out by quadrics, owing to (I). 

Then (II) would apply to all curves $C$ whose canonical model $C'$ is linearly normal, namely, to all Gorenstein curves and to all \emph{nearly Gorenstein} curves. The latter were introduced in \cite{KM}*{Dfn.~5.7} and are defined by local properties based on \cite{BF}*{p.~418}. Nearly Gorenstein curves are characterized as the only non-Gorenstein curves with $C'$ is linearly normal \cite{KM}*{Thm.~5.10} (see \eqref{defnng} and \eqref{remrel}).

 Note also that, if we disregard (iv), then the statement of (II) holds for any smooth curve, including hyperelliptic ones, thus extending the classical result. Indeed, if $C$ is hyperelliptic, then $C'$ is the rational normal curve of degree $g-1$ in $\mathbb{P}^{g-1}$ by\cite{KM}*{Thm. 3.4}, so (i) holds.   

We derive (a) and (b) above for a rational monomial curve $C$ by means of a stronger result, that is, we give a thorough description of the generators of the ideal $I(C')$ of the canonical model $C'\subset\mathbb{P}^{g-1}$. To state the result, we set some notation. Let $\gap =\{\ell_1=1,\ell_2,\ldots,\ell_{g-1},\ell_{g}=\gamma \}$ be the set of gaps in the semigroup of values of the singularity of $C$. For every $2 \leq s \leq \gamma$, consider all $s$ partitions as the sum of two gaps, say 
$$
s=a_{s_{i}}+b_{s_{i}},   \text{for} \  i=0, \ldots, \nu_{s}
$$
with $a_{s_{i}} \leq b_{s_{i}}$ and $a_{s_{0}}<a_{s_{1}} \ldots < a_{s_{\nu_{s}}}$. Set $a_{s}:=a_{s_{0}}$ and $b_{s}:=b_{s_{0}}$. In this case, we say that $a_s+b_s=s$ is a \emph{minimal partition} (of $s$). With this in mind, we have the following result:

\medskip

\noindent {\bf Theorem 2.} \emph{Let $C$ be a unicuspidal monomial curve and nearly Gorenstein. Then we can write $\mathbb{P}^{g-1}={(X_{\ell_{1}}: \ldots : X_{\ell_{g}})}$ in such a way that the statements are valid:
\begin{itemize}
\item[(i)] If $C$ is not Kunz, then the ideal of $C'$ is given by
$$
I(C')=\langle X_{a_{s}}X_{b_{s}}-X_{a_{s_{i}}}X_{b_{s_{i}}} \rangle
$$ 
to $s \ in \ \{2, \ldots, \gamma\}$ and $i \in \{1, \ldots, \nu_{s}\}$.
\item[(ii)] If $C$ is Kunz, but not trigonal with $g_3^1$ free of basis points, then the ideal of $C'$ is given by
$$
I(C') = \langle X_{a_{s}}X_{b_{s}}-X_{a_{s_{i}}}X_{b_{s_{i}}}, X_{\gamma/2}^{3}-X_{1}X_{a}X_{b}, X_{\gamma/2}^{3}-X_{a'}X_{b'}X_{\gamma}\rangle
$$
where $a+b=\gamma/2-1$ and $a'+b'=\gamma/2$ are minimal partitions.
\item[(iii)] If $C$ is trigonal with $g_3^1$ free of base points, then the ideal of $C'$ is given by
\begin{align*}
    I(C') = &\langle  X_{a_{s}}X_{b_{s}}-X_{a_{s_{i}}}X_{b_{s_{i}}},  X_{\gamma/2}^{3}-X_{1}X_{a}X_{b}, X_{\gamma/2}^{3}-X_{a'}X_{b'}X_{\gamma}, \\
    & X_{2}X_{\gamma/2}^{2}-X_{1}^{2}X_{\gamma}, X_{3k+r}X_{\gamma/2}^{2}-X_{1}X_{c}X_{d}, X_{3k+r}X_{\gamma/2}^{2}-X_{c'}X_{d'}X_{\gamma}\rangle
\end{align*}
for $1 \leq k \leq m-1$, $0 < r < 3$, $\gamma=2(3m+r)$ where $c+d=3k+r+\gamma-1$ are $c'+d'=3k+r$ minimal partitions. 
\end{itemize}}

The above result is proved in \eqref{EBmonomial}. It heavily depends on the computation of the dimensions of the homogeneous parts $I_n(C')$, which we give a proof for an arbitrary curve in \eqref{thmdim}. The techniques we use are highly inspired by those developed by Stöhr in \cite{St1} dealing with \emph{exceptional} monomials. Our approach has some differences though, which we stress in \eqref{remebm}. 

Finally, we raise a couple of questions that could be addressed in the near future: (1) Does Theorem 1.(II) hold in general as explained in \eqref{remebb}? What can be said if $C'$ is not linearly normal?; (2) Is it possible to describe the moduli of rational nearly Gorenstein curves with a prescribed semigroup at a singularity as explained in \eqref{remebm}?; (3)  For smooth curves, Enriques-Babbage Theorem can be stated as in Green's conjecture, i.e., $K_{1,2}(C,\ww)=0$ iff $\cliff(C)>1$; but Theorem 1.(II) yiedls this is false for non-Gorenstein curves. In fact,  we build in \eqref{cliffgthm2} Kunz curves of arbitrary high Clifford index, but $K_{1,2}(C,\ww)$ doesn't vanish for those curves due to Theorem 1.(I).(i) and Theorem 2. Therefore, it would be interesting to characterize curves for which $\cliff(C)=1$. A study of the Clifford index (and Clifford dimension) of an integral curve is presented in \cite{FGMS}, although this question has not been fully addressed. In \cite{FGMS}*{Prop. \,3.5 (v)}, for instance, it is shown that any trigonal integral curve has Clifford index \(1\), but there are no statements regarding sufficiency.

\

\noindent{\bf Acknowledgments.} 
This work corresponds to part of the Ph.D. Thesis \cite{At} of the fourth-named author. The third-named author is partially supported by CNPq grant number 308950/2023-2 and FAPESP grant number 2024/15918-8. The fourth-named author is supported by CAPES grant number 88887.821937/2023-00.

\section{Preliminaries}\label{prelim}

Let $C$ be an integral and projective curve of arithmetic genus $g$ defined over an algebraically closed field $k$ with structure sheaf $\oo_C$, or simply $\oo$. Let $\ww$ be the dualizing sheaf of $C$. Recall that a point $P\in C$ is \emph{Gorenstein} if the stalk $\ww_P$ is a free $\oo_P$-module,
where $\ww$ stands for the dualizing sheaf on $C$. 
The curve $C$ is said to be \emph{Gorenstein} if all of its points are Gorenstein, or equivalently,
if $\ww$ is invertible.

\begin{sbs}[Linear Systems] \label{LinSys}
 Following \cite{AK}, a \emph{linear system of degree $d$ and dimension $r$} on $C$ (referred as a $g_d^r$, for short) is a set of exact sequences, identified by a pair
$$
(\fff,V)= \{0\to \mathcal{I} \stackrel{\iota_{\lambda}}{\longrightarrow} \ww \longrightarrow \mathcal{Q}_\lambda \to 0\}_{\lambda\in V}
$$
where $\fff$ is a torsion-free sheaf of rank $1$ on $C$ with $\deg \fff :=\chi (\fff )-\chi (\mathcal{O}) = d$, and $V$ is a nonzero subspace of $H^{0}(\fff)$ of dimension $r+1$. Also, $\mathcal{I} := \mathcal{H}{\rm om}(\fff,\ww)$ and $\iota_{\lambda}:=\mathcal{H}{\rm om}(\lambda,\ww)$. Note that, if $\oo\subset\fff$, then
\begin{equation}\label{eqdega}
 \deg\fff = \sum_{P\in C} \dim\bigl(\fff_P\big/\op \bigr)
  \end{equation}

\begin{defi}
The \emph{gonality} of $C$ is the smallest $d$ for which $C$ carries a $g_d^1$, or, equivalently, for which there is a torsion free sheaf $\mathcal{F}$ of rank $1$ on $C$ with degree $d$ and $h^0(\mathcal{F})\geq 2$. We say $\fff$ \emph{contributes} to $\gon(C)$ if $h^0(\fff)\geq 2$. If, besides, $\deg(\fff)=\gon(C)$, we say $\fff$ \emph{computes} $\gon(C)$.
\end{defi}

In addition, call a point $P\in C$ a \emph{base point of $(\fff,V)$} if,
for all $\lambda\in V$, the injection $\lambda :\op\to\fff_P$ is not an
isomorphism.  Call a base point \emph{removable} if it isn't a base
point of $(\oo_C\langle V\rangle,V)$, where $\oo_C\langle V\rangle$ is the $\oo$-submodule of $\fff$ generated by $V$. We say $(\fff,V)$ is \emph{base point free} if it has no base points. If $\fff$ is invertible and it is generated by $V$, then the linear system induces a morphism which we write $(\fff,V): C\to\mathbb{P}^r$. If $(\fff,V)$ is complete, i.e., $V=H^0(\fff)$, then we write $(\fff,V)=|\fff|$.
\end{sbs}

\begin{sbs}[Canonical Model] \label{cm} Given a sheaf $\mathscr{G}$ on $C$, if $\varphi :\mathcal{X}\to C$ is a morphism from a scheme $\mathcal{X}$ to $C$, we set
$$
\oo_{\mathcal{X}}\mathscr{G}:=\varphi^* \mathscr{G}/\rm{Torsion}(\varphi^*\mathscr{G})
$$
and for each coherent sheaf $\fff$ on $C$ set 
$$
\fff^n:=\rm Sym^n\fff/\rm Torsion (\rm {Sym}^n\fff).
$$ 
Given a vector space $V \subset k(C)$ we also use the notation 
\begin{equation} \label{h0omegan}
    V^{n}:= \left \{\sum_{i=1}^{m} f_{i,1}\ldots f_{i,n} \,\bigg|\, f_{i,j} \in V, \   m\in \nn^{*} \right \}
\end{equation}
\begin{defi}
Consider the normalization map $\pi :\cb\rightarrow C$.
In \cite{R}*{p.\,188\,top} Rosenlicht showed that the linear
series $(\oo_{\cb}\ww,H^0(\ww))$
is base point free. 
He then considered the induced morphism $\psi :\cb\rightarrow{\mathbb{P}}^{g-1}$
and called its image $C':=\psi(\cb)$ the \emph{canonical model} of $C$.
\end{defi}

Rosenlicht also proved \cite{R}*{Thm.\,17}
that if $C$ is nonhyperelliptic, then the map $\pi :\cb\rightarrow C$ 
factors through a map $\pi' : C'\rightarrow C$. Set $\oo':=\pi'_*(\oo_{C'})$ in this case. Let $\widehat{C}:=\rm {Proj}(\oplus\,\ww ^n)$ be the blowup of $C$ along $\ww$ and
$\widehat{\pi} :\widehat{C}\rightarrow C$ be the natural morphism. 
Set $\widehat{\oo}=\widehat{\pi}_*(\oo_{\widehat{C}})$ and $\widehat{\oo}\ww:=\widehat{\pi}_*(\oo _{\widehat{C}}\ww)$. 
In \cite{KM}*{Dfn.\,4.9} one finds
another characterization of the canonical model $C'$, namely, it is the image of the morphism  
$\widehat{\psi}:\widehat{C}\rightarrow{\mathbb{P}}^{g-1}$ defined by the linear system 
$(\oo_{\widehat{C}}\ww,H^0(\ww))$. By Rosenlicht's Theorem, since $\ww$ is 
generated by global sections, one deduces that $\widehat{\psi}:\widehat{C}\rightarrow C'$ is an 
isomorphism if $C$ is nonhyperelliptic \cite{KM}*{Thm. 6.4}.

\medskip
The sheaf $\overline{\oo}\ww:=\pi_*(\oo_{\cb}\ww)$ can be generated by
the global sections of $\ww$, see \cite{R}*{p.\,188 top}.
Since there are only a finite number of singular
points on $C$ and the ground field is infinite, it follows that there is a
differential $\zeta\in H^0(\ww)$ such that 
$(\overline{\oo}\ww)_P=\zeta\cdot\overline{\oo}_P$ for every singular point $P\in C$, where
$\overline{\oo}:=\pi _{*}(\oo _{\cb})$. 


So we will assume, once and for all, and up to explicit mention in contrary, that $\ww$ is embedded  in the constant sheaf $\mathcal{K}$ of rational functions by means of $\zeta$. If so, for each singular point $P\in C$, there is a chain of inclusions
$$
\ccc_P\subset
\oo_P \subset \wwp \subset\widehat{\oo}_P=\op'\subset\obp
$$
where $\ccc:=\mathcal{H}\rm {om}(\overline{\oo},\oo)$ is the
the \emph{conductor} of $\overline{\oo}$ into $\oo$, and where the equality makes sense if and only if $C$ is nonhyperelliptic.

\begin{defi} \label{defnng}
Let $P\in C$ be any point. Set
$$
\eta_P:=\dim(\wwp/\op)\ \ \ \ \ \ \ \ \ \ \ \mu_P:=\dim({\widehat{\oo}_{P}}/\wwp)
$$
and also
$$
\eta:=\sum_{P\in C}\eta_P\ \ \ \ \ \ \ \ \ \ \mu:=\sum_{P\in C}\mu_P
$$
Following \cite{BF}*{pp. 418, 433, Prps. 21, 28} call $C$ \emph{Kunz} if $\eta=1$ and,
following \cite{KM}*{Dfn. 5.7}, call $C$ \emph{nearly Gorenstein} if  $\mu=1$. 
\end{defi}

 \begin{rem}
\label{remrel} 
The importance of the concepts above are summarized below:
\begin{itemize}
\item[(i)] $C$ is nearly Gorenstein if and only if it is non-Gorenstein and $C'$ is projectively normal, owing to \cite{KM}*{Thm. 6.5}. 
\item[(ii)] $P$ is Gorenstein iff $\eta_P=\mu_P=0$, and $P$ is non-Gorenstein iff $\eta_P,\mu_P>0$ owing to \cite{BF}*{p. 438 top}. Besides, if $\eta_P=1$ then $\mu_P=1$, by \cite{BF}*{Prp. 21}. In particular, a Kunz curve is as close to being Gorenstein as it gets.
\end{itemize}
\end{rem}
\end{sbs}

\begin{sbs}[Semigroup of Values] 
Now, let's fix some notation about valuations. Given a unibranch $P\in C$ and any function $x\in k(C)^*$, we denote
$$
v(x)=v_{P}(x):=v_{\pb}(x)\in\mathbb{Z}
$$
where $\pb\in\cb$ lies over $P$. The \emph{semigroup of values} of $P$ is
$$
\sss=\sss_P:=v_{P}(\op ).
$$
The set of \emph{gaps} of $\sss$ is
$$
{\rm G} :=\mathbb{N}\setminus\sss = \{\ell_1,\ldots,\ell_g\}.
$$
The value $\gamma:=\ell_g$ is called the \emph{Frobenius number} of $\sss$. We also feature two elements of $\sss$, namely:
\begin{equation}
\label{equaab}
\alpha :={\rm min}(\sss\setminus\{ 0\})\ \ \ \ \text{and}\ \ \ \beta :=\gamma+1.
\end{equation}
We have
$$
v(\ccc_P)=\{ s\in\sss\,|\, s\geq\beta\}
$$
Similarly, the invariant
$$
\delta :=\#({\rm G}) 
$$
agrees with the \emph{singularity degree} of $P\in C$, that is, 
\begin{equation}
\label{equdlt}
\delta=\dim(\obp/\op).
\end{equation} 
We define the set 
\begin{equation}
\label{equkkp}
{\rm K} :=\{ a\in\mathbb{Z}\ |\ \gamma -a\not\in\sss\}
\end{equation}
whose importance will be clear later on. Finally, given ${{\rm I}}$,${{\rm J}} \subset \mathbb{Z}$, set 
\begin{equation}\label{eqij}
    {{\rm I}}-{{\rm J}} := \{a \in \mathbb{Z}  |\ a + {{\rm J}} \subset I \}.
\end{equation}
\end{sbs}

\section{On the Ideal of the Canonical Model of a Non-Gorenstein Curve} 
\label{kcgc}
 
In this section we prove some results concerning the ideal $I(C')$ of the canonical model of a non-Gorenstein curve $C$. Rather than providing a description of the generators of $I(C')$, we are more interested on their degrees, as in \eqref{Kleiman}, and the dimension of its homogeneous components $I_n(C')$, as in \eqref{thmdim}. We start by recalling some results on Koszul cohomology.

\begin{sbs}[Koszul Cohomology]
Following \cite{Gr}*{p. 126}, let $V$ be a finite dimensional $k$-vector space. Let also $S(V)$ be the symmetric algebra, and $B=\bigoplus_{q \in \mathbb{Z}} B_{q}$ a graded $S(V)$-module. Consider the Koszul complex
\begin{equation}
   \ldots \rightarrow \bigwedge^{p+1}V \otimes B_{q-1} \xrightarrow[]{d_{p,q}}  \bigwedge^{p}V \otimes B_{q} \xrightarrow[]{d_{p+1,q-1}} \bigwedge^{p-1}V \otimes B_{q+1} \rightarrow \ldots.
\end{equation}
The \emph{Koszul cohomology groups} are defined by 
\begin{equation}
   K_{p,q}(B,V) = \dfrac{\ker(d_{p,q})}{\im(d_{p+1,q-1})}.
\end{equation}
For a free minimal resolution
\begin{equation}
\label{equrem}
   \ldots \rightarrow \bigoplus_{q \geq q_{1}} M_{1,q} \otimes S(V)(-q) \rightarrow  \bigoplus_{q \geq q_{0}} M_{0,q} \otimes S(V)(-q) \rightarrow B \rightarrow 0. 
\end{equation}
we have the \emph{Syzygy Theorem} \cite{Gr}*{Thm.~(1.b.4)}
\begin{equation} \label{KpqM}
   K_{p,q}(B,V) \simeq M_{p,p+q}(B,V).
\end{equation}
The integers $\beta_{i,j} := \dim M_{i,j}$ are called the \emph{Betti numbers} of the resolution \cite{E}*{p.~8}.

Now let $\fff$ be a torsion free sheaf of rank 1 on $C$. Taking $V = H^{0}(\fff)$ and $B = \bigoplus_{n \geq 0} H^{0}(\fff^{n})$, we set 
\begin{equation}
   K_{p,q}(C,\fff) := K_{p,q}(B,V)
\end{equation}
In other words, considering the complex
\begin{equation} \label{Koszul}
    \left( \bigwedge^{p+1}H^0(\fff) \right)\otimes H^0(\fff^{q-1}) \xrightarrow[]{\phi_{p,q}^1} \left( \bigwedge^{p}H^0(\fff) \right) \otimes H^0(\fff^{q}) \xrightarrow[]{\phi_{p,q}^2} \left(\bigwedge^{p-1}H^0(\fff) \right) \otimes H^0(\fff^{q+1})
\end{equation}
then
$$
K_{p,q}(C,\fff):= \ker(\phi_{p,q}^2)/{\rm im}(\phi_{p,q}^1)
$$

If $C$ is a non-hyperelliptic Gorenstein curve, then the above has the following geometric content. Setting $\mathbb{P}^{g-1} = \mathbb{P}(H^{0}(\omega))$ and noticing that $\omega^{n} = \oo_{C}(n)$, then \eqref{equrem} provides the resolution of $C$ viewed as a canonical curve. Moreover, if $C$ is non-Gorenstein, this extrinsic approach can still be preserved, as we will see next in \eqref{Kleiman}. Also, for an arbitrary non-Gorenstein curve, \eqref{Koszul} can also be seen from a purely algebraic point of view as considered in \cites{CFMt, GM}. For instance, it was proved in \cite{GM}*{p.~8} that if $C$ is non-hyperelliptic, then $K_{0,2}(C,\ww)=0$,  that is, Max Noether Theorem holds for any integral curve. It completed \cite{CFMt}*{Thm.~1}, where the same result was proved for curves whose non-Gorenstein points are bibranch at worst. 
\end{sbs}

\begin{sbs}[Clifford Index]
Let $\fff$ be a torsion-free sheaf of rank $1$ on $C$. The \emph{Clifford index} of $\fff$ is defined as follows
$$
     \cliff(\fff)  := \deg(\fff) - 2(h^{0}(\fff)-1).
$$
Accordingly, the \emph{Clifford index} of $C$ is
$$
     \cliff(C)  :=  \min \left\{ \cliff\,(\fff) \,|\, h^{0}(\fff) \geq 2\ \text{and}\ h^{1}(\fff)\geq 2 \right\}.
$$
Say $\fff$ \emph{contributes} to the Clifford index if $h^0(\fff)\geq 2$ and $h^1(\fff)\geq 2$. Say $\fff$ \emph{computes} the Clifford index if, besides, $\cliff(\fff)=\cliff(C)$.
\end{sbs}

\begin{sbs}[Green's Conjecture] 
In \cite{Gr}*{Conj.~(5.1)}, it was stated the \emph{Noether-Enriques-Petri conjecture} nowadays known as \emph{Green's conjecture}: if $C$ is smooth, then 
\begin{equation} \label{conjgreen}
    K_{p,1}(C,\omega) \neq 0
\Longleftrightarrow
C \ \text{has a} \ g_{d}^{r} \ \text{with} \ d \leq g-1, \ r \geq 1, \ d-2r \leq g-2-p.
\end{equation}
Using duality \cite{Gr}*{Thm.~(4.c.1)}, we may write \eqref{conjgreen} as 
$$
K_{p,2}(C,\omega) = 0 \Longleftrightarrow C \ \text{has no} \ g_{d}^{r} \ \text{with} \ d \leq g-1, \ r \geq 1, \ d-2r \leq p.
$$
which can be rephrased as 
\begin{equation} \label{green2}
    K_{p,2}(C,\omega)=0 \Longleftrightarrow \cliff(C) > p.
\end{equation}
as it is likely the way the conjecture is most known, as it was said in the Introduction. But note that the statement also makes sense for an arbitrary $C$. In fact, in \cite{CFMt}*{Thm.\,3.(iv)} one finds a family of singular curves which violates Green's Conjecture at any For $p$. 

\begin{prop} \label{Kleiman}
Let $C$ be a nearly Gorenstein curve. Then 
\begin{itemize}
    \item[(i)] $K_{1,2}(C,\ww) = 0$ if and only if $C'$ is cut out by quadrics; 
    \item[(ii)] $C'$ is cut out by quadrics and cubics; 
    \item[(iii)] if $C$ is not Kunz, then $C'$ is cut out by quadrics. 
\end{itemize}
\end{prop}

\begin{proof}
\textbf{To prove (i)}, write $\bb{P}^{g-1} =\bb{P}(H^{0}(\ww))$. As $C$ is nearly Gorenstein, then the proof of \cite{Mt}*{Thm.~2.7} and \cite{Mt}*{Rem.~2.8} yield $H^{0}(\ww^{n}) = H^{0}(\oo_{C'}(n))$ and also yield $C'={\rm Proj} \bigoplus_{n \geq 0} H^{0}(\ww^{n})$. Then, first, $K_{p,q}(C,\ww) = K_{p,q}(C',\oo_{C'}(1))$, and, also, \eqref{equrem} gives the resolution of $C'$. Thus $K_{1,2}(C,\ww)=0$ if and only if $K_{1,2}(C',\oo_{C'}(1))=0$. But by definition, $K_{1,2}(C',\oo_{C'}(1))=K_{1,2}(B,V)$, where $B = H^{0}(\oo_{C'}(1))$ and $V = \bigoplus_{n\geq 0} H^{0}(\oo_{C'}(n))$. Now $K_{1,2}(B,V) = M_{1,3}(B,V)$, by \eqref{KpqM}. But $M_{1,3}(B,V) = 0$ means we need no cubics to cut out $C'$. So if we prove (ii), then (i) follows.

\textbf{To prove (ii)}, let $\mc{I}$ be the ideal sheaf of $C'$ in $\mathbb{P}^{g-1}$. We will show that the Castelnuovo-Mumford regularity \cite{M2}*{Prop.~p.~99} of $\mc{I}$ is at most $3$, that is, $H^q(\mc
I(3-q))=0$ for $q\ge1$. We follow the argument in the proof of \cite{KM}*{Thm.~6.5}. Consider the exact sequence 
\begin{equation} \label{exseq}
    0 \rightarrow \mc{I} \rightarrow \mc O_{\bb P^{g-1}} \rightarrow \mc O_{C'} \rightarrow 0.
\end{equation}
For $q=1$, take the following exact sequence in cohomology induced by \eqref{exseq},
$$
    H^{0}(\mc O_{\bb P^{g-1}}(1)) \stackrel{\alpha}{\longrightarrow} H^{0}(\mc O_{C'}(1)) \longrightarrow
 H^1(\mc I(1)) \stackrel{\beta}{\longrightarrow}  H^1(\mc O_{\bb P^{g-1}}(1)).
$$
As $C$ is nearly Gorenstein, $C'$ is projectively normal \eqref{remrel}(i). Therefore, $\alpha$ is surjective, thus $\beta$ is injective. But $g \geq 3$, thus $H^{1}(\oo_{\bb{P}^{g-1}}(1))=0$. Hence $H^{1}(\mc{I}(1))=0$. For $q \geq 2$, consider the exact sequence in cohomology induced by \eqref{exseq}, 
$$
 H^{q-1}(\mc O_{C'}(3-q))\to
 H^q(\mc I(3-q))\to H^q(\mc O_{\bb P^{g-1}}(3-q)).
$$
For $q=2$, we have that $H^{1}(\mc O_{C'}(1))=0$ because $\deg (\oo_{C'}(1)) = 2g'+\eta>2g'-2$, where the equality holds by \cite{KM}*{Prp.\,2.14}. For $q \geq 3$, we have $H^{q-1}(\oo_{C'}(3-q))=0$, because $C'$ is a curve. On the other hand, if $q \neq g$ then it is easily seen that $H^{q-1}(\oo_{\bb{P}^{g-1}}(3-q))=0$; otherwise, $H^{q-1}(\oo_{\bb{P}^{q-1}}(3-q))=H^{0}(\oo_{\bb{P}^{q-1}}(-3))=0$. Therefore, $H^{q}(\mc{I}(3-q))=0$ for $q \geq 2$ as well. This proves (ii).

\textbf{To prove (iii)}, if $C$ is not Kunz, then $\eta\geq 2$; thus $\deg(C')\geq 2g'+2$, and (iii) follows by \cite{F2}*{Cor.~1.14}.
\end{proof}
\end{sbs}

\begin{sbs}[Dimensional Counts for the Ideal of the Canonical Model]
To begin with, we start by a general result.

For any $P\in C$, set
$$
\sigma_{n,P}:=
\begin{cases}
\dim(\ww_P^n/\ww_P) &\mbox{if } P \ \text{is non-Gorenstein,} \\ 
0 & \mbox{otherwise}
\end{cases}
$$
and let $\sigma_n:=\sum_{P\in C} \sigma_{n,P}$. Let also
$$
I_n(C'):=\{\text{hypersurfaces of degree $n$ in $\mathbb{P}^{g-1}$ containing $C'$}\}.
$$
With this in mind we prove the following result.

\begin{thm}
\label{thmdim}
If $C$ is non-hyperelliptic and $n\geq 2$. Then
$$
\dim(I_n(C'))=\binom{g+n-1}{n}-n(2g-2)+(n-1)\eta-\sigma_{n}-1+g.
$$
\end{thm}

\begin{proof}
Recall that $C'$ is defined as follows: the linear series $(\oo_{\cb}\ww,H^0(\ww))$ yields a morphism $\psi :\cb\rightarrow{\mathbb{P}}^{g-1}=\mathbb{P}(H^0(\ww))$ and $C':=\psi(\cb)$. Now consider the natural morphisms
$$
\varphi_n: {\rm Sym}^n H^0(\ww) \longrightarrow H^0(\ww^n)
$$
for $n\geq 1$. Since $C$ is non-hyperelliptic, then $C'$ is birationally equivalent to $C$. Thus the very definition of $C'$ yields that $I_n(C')=\ker(\varphi_n)$, for $n\geq 1$.

Now, by the singular version of Noether Theorem, proved in \cite{GM}, we have that $\varphi_{n}$ is surjective for any $n\geq 1$. Therefore for $n \geq 2$,
\begin{align*}
\dim(I_n(C'))&=\dim({\rm Sym}^n H^0(\ww))-h^0(\ww^n)\\
             &=\dim({\rm Sym}^n H^0(\ww))-(\deg(\ww^{n})+1-g) \\
             &=\dim({\rm Sym}^n H^0(\ww))-(n(2g-2-\eta)+\eta+\sigma_{n}+1-g) \\
             &=\binom{g+n-1}{n}-n(2g-2)+(n-1)\eta-\sigma_{n}-1+g
\end{align*}
where the second equality owes to Riemmann-Roch and the fact that $h^{1}(\ww^{n})=0$ for $n\geq 2$ since $\deg(\ww^{n})>2g-2=\deg(\ww)$. For the third equality, let $U$ be the largest open set where $\ww$ is a bundle, that is, where the points of $C$ are Gorenstein; then, note that $\deg_{C\setminus U}(\ww)=\eta$, and hence $\deg_{U}(\ww)=2g-2-\eta$; thus it is straightforward that $\deg_{C\setminus U}(\ww^n)=\eta+\sigma_n$ while $\deg_{U}(\ww^n)=n(2g-2-\eta)$.
\end{proof}
\end{sbs}

\section{The Ideal of the Canonical Model of Monomial Curves} \label{icmmc}

In this section, we prove Theorem 1.(II) and Theorem 2, stated in the Introduction. We start by defining \emph{exceptional triples} of gaps of a given semigroup, very inspired by the \emph{exceptional monomials} introduced by St\"ohr in \cite{St1}*{p.\,196}.

\begin{defi} \label{excmon}
Let ${\rm A} \subset \mathbb{N}^{*}$. We say $(b_{1}, b_{2}) \subset {\rm A}^{2}$ is \emph{minimal} with respect to ${\rm A}$ if it is the smallest partition, in the lexicographic order, of $b_{1}+b_{2}$ by elements of ${\rm A}$. Now let $\sss$ be a semigroup and $\gap$ be its set of gaps. 
We say that an increasingly ordered $n$-tuple $(b_{1}, b_{2}, b_{3})\subset \gap^3$ is \emph{exceptional} if, for every $i\neq j \in \{1, 2, 3 \}$, we have that $(b_{i}, b_{j})$ is minimal with respect to $\gap$ and $b_{i}\neq 1,\gamma$.
\end{defi}

For the next result, we need the following, call . 

\begin{defi} \label{Kunz}
Let $\sss$ be a numerical semigroup. Rephrasing Barucci-Fr\"oberg \cite{BF}*{pg.~420, top}, call $\sss$ pseudo-symmetric if $\kk=\sss\cup\{\gamma/2,\gamma\}$. In particular, $\gamma$ is even (and the conductor number $\beta$ is odd).
\end{defi}


\begin{lem} \label{exclem}
Let $\sss=\{s_0=0 < s_1=\alpha < s_2 < \ldots < s_n <\ldots \}$ be a semigroup. Let also $\tau$ be the largest integer such that $s_{\tau}=\tau\alpha$ and $m$ the largest integer such that $S_m:=[\gamma-m\alpha+1,\gamma-1-(m-1)\alpha]\subset \sss$. Then the exceptional triples are
\begin{itemize}
\item[(i)] $(a_{1},\alpha-1,a_{3})$ with $a_{1}\geq 2$, $[\alpha+1,\alpha+a_1-2] \subset \sss$,

 \ \ \ \ \ \ \ \ \ \ \ \ \ \ \ \ \ \ \ \  $[a_{3}+1, a_{3}+\alpha-2]\subset \sss$, and $[a_{3}+1, a_{3}+a_1-1] \subset \sss$;
\item[(ii)] $(2,k\alpha-1,\gamma-(m+1)\alpha+1) \ \text{with} \ m\geq 1 \ ,k \in [2,m+1]$;

\ \ \ \ \ \ \ \ \ \ \ \ \ \ \ \ \ \ \ \ \ \ \ \ \ \ \ \ \ \ \ \ \ \ 
$[\gamma-(m+1)\alpha+2,\gamma-(m+1)\alpha+\alpha-1] \subset \sss$;
\item[(iii)] $(a_{1},a_{2},\gamma - a_2) \ \text{with} \ a_{1}<\alpha, \ a_{2}>\tau\alpha$,

\ \ \ \ \ \ \ \ \ \ \ \ \ \ \ \ \ \ \ \ \ \ 
$a_{2} \text{ is the smallest    gap such that } \gamma-a_{2} \in \gap$, $\text{and}$ 

\ \ \ \ \ \ \ \ \ \ \ \ \ \ \ \ \ \ \ \ \ \ 
$[a_{i}+1, a_{i}+a_1-1] \subset \sss, \ i=2,3$;
\item[(iv)] $(k\alpha+r,\gamma/2,\gamma/2)$ with $1 \leq k \leq m-1$, $0 < r < \alpha$, $\gamma=2(m\alpha+r)$ and

\ \ \ \ \ \ \ \ \ \ \ \ \ \ \ \ \ \ \ \ \ \
$\sss$ is pseudo-symmetric.
\item[(v)] $(\gamma/2,\gamma/2,\gamma/2)$ and $\sss$ is pseudo-symmetric.
\end{itemize}
\end{lem}
\begin{proof}
Let $(b_1,b_2)\in\gap^2$ be a pair. Note that
\begin{equation} \label{propminimal}
(b_1,b_2) \ \text{is minimal} \  \Longleftrightarrow \ (b_{1}-i \in\gap \Longleftrightarrow b_2+i \not\in\gap, \ \forall \ 1\leq i \leq b_1-1).
\end{equation}
Also, 
assume $(a_{1},a_{2},a_{3})$ is exceptional.

\noindent {\bf case 1:} $2\leq a_1 \leq a_{2} \leq \alpha-1$.

\noindent As $(a_{1},a_{2})$ is minimal, then \eqref{propminimal} implies that $a_{2}=\alpha-1$ and $[\alpha, \alpha+a_{1}-2] \subset \sss$.
Also, $(a_{1},a_{3})$ and $(\alpha-1,a_3)$ are minimal too. So it follows, respectively, that $[a_{3}+1, a_{3}+a_{1}-1] \subset \sss$ and $[a_{3}+{1}, a_{3}+\alpha-2] \subset \sss$ by \eqref{propminimal}. This yields (i).  

\noindent {\bf case 2:} $2\leq a_1 < \alpha$, and $\alpha < a_{2} < \tau\alpha$.

\noindent As $(a_1,a_2)$ is minimal, \eqref{propminimal} yields $a_1=2$ and $a_2=k\alpha-1$ for some $k\leq\tau$. Now note that if a block $B$ of $\alpha-1$ integers is in $\sss$, so is $B+n\alpha$, for every $n\geq 1$. Thus, as $(k\alpha-1,a_3)$ is minimal, we must have $a_{3}=\gamma-(m+1)\alpha+1$, and necessarily $k\leq m+1$, $k \in [2,m+1]$, and $[\gamma-(m+1)\alpha+2,\gamma-(m+1)\alpha+\alpha-1] \subset \sss$. This yields (ii).


\noindent {\bf case 3:} $2\leq a_1 < \alpha$, and $a_{2} > \tau\alpha$.

\noindent Note that if $(a_2,a_3)$ is minimal, there is no other possibility than $m=\tau$, $a_3=\gamma-a_2$, and for any $n\in\{1,\ldots,a_3-\tau\alpha\}$, we have $a_3+i\in \sss$ if and only if $a_2-i\in\gap$. In particular, $a_{2}$ is the smallest gap such that $\gamma-a_2\in \gap$. Now $(a_1,a_2)$ and $(a_1,a_3)$ are also minimal and hence 
$[a_{i}+1, a_{i}+a_1-1] \subset \sss$, for $i=2,3$. This yields (iii).

\noindent {\bf case 4:} $\alpha<a_1<\tau\alpha$.

\noindent Write $a_1=k\alpha+r$ with $0<r<\alpha$. Assume $a_2<\tau\alpha$. As $(a_1,a_2)$ is minimal, \eqref{propminimal} yields either: (a) $r=1$ and $a_2=\tau\alpha-2$ or, else, (b) $a_2=\tau\alpha-1$. Assume (a). Then, as $(a_1,a_{3})$ is minimal, we get $a_{3}+1 \in \gap$. On the other hand, as $(a_{2},a_{3})$ is minimal, we get $a_{3}+1 \in \sss$. So (a) is precluded. Now assume (b). Then, as $(a_{2},a_{3})$ is minimal, we get $a_{3}=\gamma-(m+1)\alpha+1$ and $[\gamma-(m+1)\alpha+2,\gamma-m\alpha-1] \subset \sss$. Now, as $(a_{1},a_{3})$ is minimal, then $r=\alpha-1$. Further, as $(a_{1},a_{2})$ is minimal, there is no possibility left other than $a_{3}=a_{2}=\gamma/2$ and $\sss$ is pseudo-symmetric. Then, as $a_{3}=\gamma/2$ and $r=\alpha-1$ it follows that $\gamma=2(m\alpha+r)$.

Otherwise, assume $a_{2}>\tau\alpha$. As $(a_2,a_3)$ is minimal, case 3 yields $m=\tau$, $a_3=\gamma-a_2$, and for any $n\in\{1,\ldots,a_3-\tau\alpha\}$, we have $a_3+i\in \sss$ if and only if $a_2-i\in\gap$. Now, write $a_{1}=k\alpha+r$, with $0 < r < \alpha$. As $(a_1,a_3)$ is minimal, then $a_{3} = \gamma-m\alpha-r$. As $(a_{1},a_{2})$ is minimal, again we have that there is no possibility left other than $a_{3}=a_{2}=\gamma/2$ and $\sss$ is pseudo-symmetric. Then, as $a_{3}=\gamma/2$, it follows that $a_{2}=\gamma-m\alpha+r$. Now, we impose the condition $k \leq m-1$ for $a_{1}$ to be different than $a_{2}$ and $a_{3}$. This yields (iv).

\noindent {\bf case 5:} $a_1>\tau\alpha$.

\noindent Replacing $(a_2,a_3)$ by $(a_1,a_2)$ in case 3 yields $m=\tau$, $a_2=\gamma-a_1$, and  $a_{1}$ is the smallest gap such that $\gamma-a_1\in \gap$. But as $a_2+a_3\leq \gamma$. This yields $a_1=a_2=a_3=\gamma/2$ and $\sss$ is pseudo-symmetric. This yields (v).
\end{proof}
\begin{sbs}[The Ideal of the Canonical Model of Nearly Gorenstein Curves]
Now we want to describe the ideal of nearly Gorenstein unicuspidal monomial curves. To begin with, we start by characterizing such a curve in terms of its semigroup of values. 
\begin{lem} \label{ngumc}
Let $C$ be a unicuspidal monomial curve with semigroup $\sss$. Let $\kk$ be as in \eqref{equkkp}, and $\gamma$ be the Frobenius number. Then $C$ is nearly Gorenstein if and only if $\langle \kk \rangle = \kk \cup \{\gamma\}$, where $\langle \kk \rangle$ is the semigroup generated by $\kk$. Also, $C$ is Kunz if and only if $\sss$ is pseudo-symmetric.
\end{lem}
\begin{proof}
Assume $C$ is nearly Gorenstein. Recall by \eqref{defnng} this is equivalent to saying that $\mu = 1$. In our case it corresponds to the equality $\dim (\widehat{\mathcal{O}}_{P}/\omega_{P})=1$, i.e., $\# (v_{\pb}(\widehat{\oo}_{P})\setminus v_{\pb}(\wwp))=1$. Now, $v_{\pb}(\wwp)=\kk$; on the other hand, by the very definiton, $\widehat{C} = \mathcal{B}\ell_{\ww}C$. Therefore $\widehat{\oo}_{P}$ is the smallest subring of $\overline{\oo}_{P}$ containing $\wwp$. As $C$ is monomial, $v_{\pb}(\widehat{\oo}_{P})= \langle \kk \rangle$. But $\gamma \in \langle \kk \rangle$, thus $C$ is nearly Gorenstein if and only if $\langle \kk \rangle=\kk \cup \{ \gamma \}$. The last assertion is immediate from \eqref{remrel} and \eqref{Kunz}.
\end{proof}
Let $C$ be a unicuspidal monomial curve with semigroup $\sss$ whose set of gaps is $\gap$. We have seen above that 
\begin{equation*}
H^{0}(\omega)=\{t^i\,|\,i\in\gamma-\gap\}
\end{equation*}
and that the canonical model of $C$ is hence
\begin{equation*}
C'=(1:t^{b_2}:\ldots :t^{b_{\delta}})
\end{equation*}
where $\{0,b_2,\ldots,b_{\delta}\}=\gamma-\gap$.

So write $\gap =\{\ell_1=1,\ell_2,\ldots,\ell_{g-1},\ell_{g}=\gamma \}$, and 
\begin{equation*} 
    \mathbb{P}^{g-1}=\{(X_{\ell_{1}}: \ldots : X_{\ell_{g}})\}=\{(t^{\gamma-1}:t^{\gamma-\ell_2}:\ldots:t^{\gamma-\ell_{g-1}}:1)\}=\mathbb{P}(H^0(\ww)).
\end{equation*}

Also, for each $2\leq s \leq \gamma$ we may write $s = a+b$, where $a,b \in \gap$.
Now we consider all partitions of $s$ as a sum of two gaps, say 
$$
s=a_{s_{i}}+b_{s_{i}}, \ \ \text{for}\ \ i=0, \ldots, \nu_{s}
$$
with $a_{s_{i}} \leq b_{s_{i}}$ and $a_{s_{0}}<a_{s_{1}} \ldots < a_{s_{\nu_{s}}}$. Set $a_{s}:=a_{s_{0}}$ and $b_{s}:=b_{s_{0}}$. Note that $(a_s,b_s)$ is minimal for every $2\leq s\leq \gamma$.
\begin{thm} \label{EBmonomial} Let $C$ be a nearly Gorenstein unicuspidal monomial curve. Then:
\begin{itemize}
\item[(i)] If $C$ is not Kunz, then the ideal of $C'$ is given by 
$$
I(C')=\langle X_{a_{s}}X_{b_{s}}-X_{a_{s_{i}}}X_{b_{s_{i}}} \rangle
$$ 
for $s \in \{2, \ldots, \gamma\}$ and $i \in \{1, \ldots, \nu_{s}\}$.
\item[(ii)] If $C$ is Kunz, but not trigonal computed by a base point free pencil, then the ideal of $C'$ is given by 
$$
I(C') = \langle  X_{a_{s}}X_{b_{s}}-X_{a_{s_{i}}}X_{b_{s_{i}}}, X_{\gamma/2}^{3}-X_{1}X_{a}X_{b}, X_{\gamma/2}^{3}-X_{a'}X_{b'}X_{\gamma}\rangle.
$$
where $(a,b)$ is the minimal decomposition of $3\gamma/2-1$ and $(a',b')$ is the minimal decomposition of $\gamma/2$.
\item[(iii)] If $C$ is trigonal computed by a base point free pencil then the ideal of $C'$ is given by
\begin{align*}
    I(C') = &\langle  X_{a_{s}}X_{b_{s}}-X_{a_{s_{i}}}X_{b_{s_{i}}},  X_{\gamma/2}^{3}-X_{1}X_{a}X_{b}, X_{\gamma/2}^{3}-X_{a'}X_{b'}X_{\gamma}, \\
    & X_{2}X_{\gamma/2}^{2}-X_{1}^{2}X_{\gamma}, X_{3k+r}X_{\gamma/2}^{2}-X_{1}X_{c}X_{d}, X_{3k+r}X_{\gamma/2}^{2}-X_{c'}X_{d'}X_{\gamma}\rangle.
\end{align*}
with $1 \leq k \leq m-1$, $0 < r < 3$, $\gamma=2(3m+r)$, and $(c,d)$ is the minimal decomposition of $3k+r+\gamma-1$ and $(c',d')$ is the minimal decomposition of $3k+r$. 
\end{itemize}
\end{thm}

\begin{proof} Set $\Gamma_{n} = \{ a_{1} + \ldots + a_{n} \, | \, a_{i} \in \gamma-\gap \} $ and assume $C$ is non-Gorenstein. We claim that, for $n\geq 2$,
\begin{equation} \label{vn}
    H^0(\ww^n) = V_{n}:=\bigg\langle t^{i} \, \big{|}\, i \in \Gamma_{n} \cap [0,\gamma-1] \bigg\rangle \bigcup \bigg\langle t^{i} \, \big{|} \, i \in [\gamma, n(\gamma-1)] \cap \nn \bigg\rangle.
\end{equation}
The assumption that $C$ is non-Gorenstein is used for the fact that $t^{\gamma}\in H^0(\ww^n)$ for any $n\geq 2$ since there are $\ell_i$ and $\ell_j$ with $\ell_i+\ell_j=\gamma$. If so, $t^{\gamma}=t^{\gamma-\ell_i}\times t^{\gamma-\ell_j}\times 1^{n-2}$ which we will see right away that is in $H^0(\ww^n)$.
To prove the claim, note that $t^{i} \in H^{0}(\ww^{n})$ if and only if $t^{i} \in \ww_{R}^{n}$ for all $R \in C$. Now, $\ww_{Q}^{n} = t^{-n(\gamma-1)}$, so $V_{n}\subset\ww_R^n$ as all elements in $V_{n}$ have order at most $n(\gamma-1)$. If $R \neq P,Q$, then $\ww_{R}^{n} = \oo_{R}$, thus $V_{n} \subset \ww_{R}^{n}$ as well. Moreover, $V_{n} \subset \ww_{P}^{n}$ owing to the following equality
$$
\wwp^{n} = H^{0}(\ww)^n + \cp
$$
which can be easily derived from the proof of \cite{KM}*{Lem.~6.1}; also recall the notation \eqref{h0omegan} for the first summand in the right hand side of the equality. Thus, $V_{n} \subset H^{0}(\ww^{n})$. Conversely, by \eqref{thmdim}, $h^{0}(\ww^{n}) = n(2g-2)-(n-1)\eta+\sigma_{n}+1-g$. 
\begin{align*}
 \dim V_{n} & = \# \{\Gamma_{n} \cap [0,\gamma-1] \} +  \# ( [\gamma, n(\gamma-1)] \cap \nn) \\
 &= (g+\sigma_{n}-1)+((n-1)\gamma-n+1) \\
 &= (g+\sigma_{n}-1)+((n-1)(2(g-\eta)+\eta-1)-n+1) \\
 &= n(2g-2)-(n-1)\eta+\sigma_{n}+1-g
\end{align*}
as expected, so $H^{0}(\ww^{n})=V_{n}$. Now, ${\rm Sym}^n H^0(\ww) \twoheadrightarrow H^{0}(\ww^{n})$. In particular. 
\begin{equation} \label{gamma2}
    \{ \gamma, \ldots, 2(\gamma-1) \} \subset \Gamma_{2}
\end{equation}
also, for $n \geq 3$,
\begin{equation} \label{gamma3on}
    \{ (n-1)(\gamma-1)+1, \ldots, n(\gamma-1) \} = \{\gamma, \ldots, 2\gamma-1 \} + (n-2)(\gamma-1) \subset \Gamma_{n}
\end{equation}
because clearly $\gamma-1 \in \gamma-\gap$. 

By \eqref{gamma2}, for every $i\in  \{ \gamma, \ldots, 2(\gamma-1) \}$, we have that $i=(\gamma-\ell_j)+(\gamma-\ell_k)$ for some $1\leq j,k\leq g$. Equivalently, for each $2\leq s \leq \gamma$ we may write $s = a+b$, where $a,b \in \gap$.
Now we consider all partitions of $s$ as a sum of two gaps, say 
$$
s=a_{s_{i}}+b_{s_{i}}, \ \ \text{for}\ \ i=0, \ldots, \nu_{s}
$$
with $a_{s_{i}} \leq b_{s_{i}}$ and $a_{s_{0}}<a_{s_{1}} \ldots < a_{s_{\nu_{s}}}$. Set $a_{s}:=a_{s_{0}}$ and $b_{s}:=b_{s_{0}}$. Note that $(a_s,b_s)$ is minimal for every $2\leq s\leq \gamma$. Set also
$$
A_n:=(\Gamma_n\setminus\Gamma_{n-1})\cap[0,\gamma-1]
$$

Now, for any $\ell\in \gap$, set $x_{\ell}:=t^{\gamma-\ell}\in k(C)$. We have that \eqref{vn}, \eqref{gamma2}, \eqref{gamma3on} yields that we can write down a basis of $H^{0}(\ww^{n})$ as 
$$
\begin{cases}
x_{\ell_{j}}x_{\gamma}^{n-1} &\mbox{\text{for}} \ 1 \leq j \leq g \\
x_{i_1} \ldots x_{i_{k}}x_{\gamma}^{n-k} &\mbox{for} \ 2 \leq k \leq n, \ k\gamma-\sum_{j=1}^{k}i_{j} \in A_{k}, \ \text{with} \ (i_{1}, \ldots, i_{k}) \ \text{minimal}, \\
x_{1}^{i}x_{a_{s}}x_{b_{s}}x_{\gamma}^{n-2-i} &\mbox{\text{for}} \ 0 \leq i \leq n-2, \ 2 \leq s \leq \gamma.
\end{cases}
$$
where the first line provides a basis for $H^{0}(\ww)$, the third line corresponds to the second component of the union in \eqref{vn} and the second line are the remaining elements in the first component of the union in \eqref{vn}.

Let $\Lambda_{n} \subset k[X_{\ell_{1}}, \ldots, X_{\ell_{g}}]_{n}$ be the vector space generated by 
\begin{equation} \label{basisng}
\begin{cases} 
X_{\ell_{j}}X_{\gamma}^{n-1} &\mbox{\text{for}} \ 1 \leq j \leq g \\
X_{i_1} \ldots X_{i_{k}}X_{\gamma}^{n-k} &\mbox{for} \ 2 \leq k \leq n, \ k\gamma-\sum_{j=1}^{k}i_{j} \in A_{k}, \ \text{with} \ (i_{1}, \ldots, i_{k}) \ \text{minimal}, \\
X_{1}^{i}X_{a_{s}}X_{b_{s}}X_{\gamma}^{n-2-i} &\mbox{\text{for}} \ 0 \leq i \leq n-2, \ 2 \leq s \leq \gamma.
\end{cases}
\end{equation}
We claim that
\begin{equation} \label{somadireta}
    k[X_{\ell_{1}}, \ldots, X_{\ell_{g}}]_{n} = I_{n} \oplus \Lambda_{n}
\end{equation}
where $I_{n}:=I_{n}(C')$. Indeed, note that ${\rm Sym}^n H^0(\ww)=k[X_{\ell_{1}}, \ldots, X_{\ell_{g}}]_{n}$. Also, the morphisms ${\rm Sym}^n H^0(\ww) \rightarrow H^{0}(\ww^{n})$, if restricted to $\Lambda_n$, provide, by construction, linear isomorphisms $\Lambda_n\simeq H^0(\ww^n)$. As the $\varphi_{n}$ are surjective and $I_{n}(C')=\ker(\varphi_{n})$, the claim follows.

Now we assume that $C$ is nearly Gorenstein. This means, in our case, that $\dim({\widehat{\oo}_{P}}/\ww_{P})=1$. But we have
\begin{equation}
    \oo_{P} \subset \ww_{P} \subset \ww_{P}^{2} \subset \ldots \subset \ww_{P}^{n} \subset \ldots \subset \widehat{\oo}_{P}
\end{equation}
and also $t^{\gamma}\in \widehat{\oo}_{P}\setminus\ww_{P}$. Therefore $\ww_{P}^{n} = \langle t^{\gamma}, \gamma-\gap \rangle$ for all $n\geq 2$. In particular, $A_n=\emptyset$ for all $n\geq 2$, and the second line in \eqref{basisng} does not exist. 

To prove our result by \eqref{Kleiman}, it suffices to show that $I_{3}$ is contained in the ideal, say $I$, described in (i) and (ii).

So let $(i,j,k),(i',j',k')\in\gap^3$ be two decompositions of a same number in $\Gamma_3$. For shortly, we use $(i,j,k)\sim(i',j',k')$ when so. If the triples have a common coordinate, say $k=k'$, then $X_{i}X_{j}X_{k} \equiv X_{i'}X_{j'}X_{k} \mod I$. Indeed, write $s=i+j$. Then,
\begin{align*}
     X_{i}X_{j}X_{k} - X_{i'}X_{j'}X_{k} & = X_{a_{s}}X_{b_{s}}X_{k}-X_{i'}X_{j'}X_{k}-(X_{a_{s}}X_{b_{s}}X_{k}-X_{i}X_{j}X_{k}) \\
     & = (X_{a_{s}}X_{b_{s}}-X_{i'}X_{j'})X_{k}-(X_{a_{s}}X_{b_{s}}-X_{i}X_{j})X_{k}
\end{align*}
which is an element of $I$. 

So given $F=X_{i}X_{j}X_{k}\in k[X_{\ell_{1}}, \ldots, X_{\ell_{g}}]_{3}$, it suffices to show that $(i,j,k)$ is equivalent to either $(1,a_s,b_s)$ or $(a_s,b_s,\gamma)$ for some $s \in \{2, \ldots, \gamma \}$ owing to \eqref{basisng} and \eqref{somadireta}. If $i=1$ (resp. $k=\gamma$) then $(i,j,k) \sim (1,a_s,b_s)$ where $s=j+k$ (resp. $(i,j,k)\sim (a_s,b_s,\gamma)$ where $s=i+j$), and we are done. Otherwise, assume, first that $(i,j,k)$ is not exceptional, in the sense of \eqref{excmon}. Then one of the pairs $(i,j)$, $(i,k)$ or $(j,k)$ is not minimal. Say the last one is such and write $j+k=s$. Then $(i,j,k)\sim (i,a_s,b_s)$ for $s=j+k$. If $(i,a_{s},b_{s})$ is not exceptional, repeat the procedure. After a finite number of steps we get that $(i,j,k) \sim (i',j',k')$ where either $i'=1$ or $k'=\gamma$, or $(i',j',k')$ is exceptional. Now the exceptional triples are described in Lemma \eqref{exclem}. Let us address its five cases.

{\bf Step 1:} $(a_1,\alpha-1,a_3)$ is as in case (i). 

\noindent Write $a_3=\gamma-s\alpha+p$ for some $s\geq 1$ and $p\in\{0,1\}$. Write also $n_{\tau+1}=\tau\alpha+r$. 

\noindent If $\tau\geq 2$ or $r\neq 1$, thus $a_{1}=2$. If $p=1$, then 
$$
(2,\alpha-1,a_3) = (2,\alpha-1, \gamma-s\alpha+1)
                 \sim (3,\alpha-1, \gamma-s\alpha)
                 \sim (1,\alpha+1, \gamma-s\alpha) 
$$
or, else, if $p=0$, then $s \leq \tau$. If, moreover, $s\geq 2$, then 
\begin{align*}
(2,\alpha-1,a_3) &\sim (2,\alpha-1+a_3-(\gamma-n_{\tau+1}), \gamma-n_{\tau+1}) \\
                 & = (2,\alpha-1+\gamma-s\alpha-\gamma+\tau\alpha+r, \gamma-n_{\tau+1}) \\
                 & = (2,(\tau+1-s)\alpha+r-1, \gamma-n_{\tau+1}) \\
                 & \sim (1,(\tau+1-s)\alpha+r, \gamma-n_{\tau+1}) 
\end{align*} 
and we are done since $1\leq \tau+1-s< \tau$, so the second component is a gap. 

If $r=\tau=1$, then $a_3=\gamma-b\alpha$ for $b=1$ or $2$. If $b=2$, then $[\alpha+2,2\alpha-1]\subset\gap$. Thus $a_1=2$ or $3$. Also, $[\gamma-2\alpha-2,\gamma-2\alpha]\subset\gap$ as well. Therfore
$$
(a_1,\alpha-1,\gamma-2\alpha) \sim (4,\alpha-1, \gamma-2\alpha-(4-a_1)) \sim (1,\alpha+2,\gamma-2\alpha-(4-a_1))  
$$
If $a_3=\gamma-\alpha$, let $e$ be the largest integer such that $[\alpha,\alpha+e]\subset\sss$. Then, clearly,  $[\gamma-\alpha-e,\gamma-\alpha]\in\gap$, and also $a_1\leq e+2$ since $(a_1,\alpha-1)$ is minimal. Assume $e<\alpha-3$. Then $e+3\in\gap$. If $a_1\geq 3$, then
\begin{align*}
(a_1,\alpha-1,\gamma-\alpha) &\sim (e+3,\alpha-1, \gamma-\alpha-(e+3-a_1)) \\
                             &\sim (1,\alpha+e+1, \gamma-\alpha-(e+3-a_1)) 
\end{align*} 

So we are reduced to analyze the case $(a_1,\alpha-1,\gamma-\alpha)$ where either $a_1=2$ or, else, $\tau=1$ and $e\geq\alpha-3$. 


Assume $a_1=2$. As $C$ isn't Gorenstein, $\sss$ is non-symmetric. So we can pick $(a,b)\in\gap^2$ such that $a+b=\gamma$. Assume also $(a,b)$ is minimal. Note that $a>\alpha$. Indeed, the minimality of $(\alpha-1,\gamma-\alpha)$ implies $[\gamma-\alpha+1, \gamma-2]\subset\sss$. So if $a<\alpha$, then $a=1$. But, as $C$ is nearly Gorenstein, $\sss$ is almost symmetric. Then $2\in\kk$, and if so, $\gamma-2\in\gap$, a contradiction.

Now let $c\in\nn^*$ be such that $b-c$ is the greatest gap smaller than $b$. Plainly, $c<\alpha$, as otherwise $[a,a+\alpha-1]\subset\gap$ which cannot occur as there are no $\alpha$ consecutive integers in $\gap$. So, first, assume $c<\alpha-1$. If so,
\begin{align} \label{newtriple}
    (2, \alpha-1,\gamma-\alpha) &\sim (2,\alpha-1+(a-\alpha+c),\gamma-\alpha-(a-\alpha+c)) \\
    &= (2,a+c-1,b-c) \sim (1+c,a,b-c) \sim (1,a,b). \nonumber
\end{align}
Note $a+c-1\in\gap$ as $\gamma-(a+c-1)=b-c+1\in\sss$. Also, $1+c\in\gap$ as $c<\alpha-1$.

Now assume $c=\alpha-1$. Then, clearly, $a=m\alpha+1$ and $b=\gamma-m\alpha-1$. Thus
\begin{align} \label{newtriple2}
    (2, \alpha-1,\gamma-\alpha) &\sim (2,\alpha-1+(m-1)\alpha,\gamma-\alpha-(m-1)\alpha) = (2,m\alpha-1,\gamma-m\alpha) \nonumber \\
    & \sim (3,m\alpha-1,\gamma-m\alpha-1)  \sim (1,m\alpha+1,\gamma-m\alpha-1). \nonumber
\end{align}

If $e \geq \alpha-3$ and $\tau=1$, suppose first that $\alpha \geq 5$. Then either, $\gamma=2\alpha-1$ or $\gamma=2\alpha-2$. If $\gamma=2\alpha-1$, then $\gap\setminus([1,\alpha-1]\cup\{\gamma\})\subset\{\gamma-1\}$; if $\gamma-1\in\gap$ then $\sss$ is not almost symmetric, and if $\gamma-1\not\in\gap$ then $\sss$ is symmetric, so both are precluded by our hypothesis on $C$ being nearly Gorenstein. While if $\gamma=2\alpha-2$, then $C$ is Kunz, for which the canonical model $C'$ is not cut out by quadrics as we will see right away.  

{\bf Step 2:} $(2,k\alpha-1,\gamma-(m+1)\alpha+1)$ as in case (ii). 

\noindent First, note that the case $3 \in \sss$ is equivalent to saying that $C$ is trigonal computed by a base point free pencil by \eqref{trigonalchar}. Moreover, $C$ is Kunz if so by the same result. So we will address this case at the end of this proof. Thus we may assume $3\in\gap$. Then
\begin{align*}
(2,k\alpha-1,\gamma-(m+1)\alpha+1) &\sim (3,k\alpha-2,\gamma-(m+1)\alpha+1) \\
                 & \sim (3,\alpha-1, \gamma-(m+1)\alpha+1+(k-1)\alpha-1) \\
                 & = (3,\alpha-1, \gamma-(m+2-k)\alpha) \\
                 & \sim (1,\alpha+1, \gamma-(m+2-k)\alpha).
\end{align*}

{\bf Step 3:} 
$(a_{1},a_{2},\gamma - a_2)$ with $a_{1}<\alpha, \ a_{2}>\tau\alpha$ as in case (iii). 

\noindent If $a_1=\alpha-1$, then $a_2=k\alpha-1$ and $a_3=\gamma-k\alpha+1$ for $k\geq 2$ since $a_2>\tau\alpha$. Also, $[\gamma-s\alpha-\alpha+1,\gamma-s\alpha-1]\subset\sss$ for $s\leq k$ because $(\alpha-1,a_3)$ is minimal. Thus
\begin{align*}
    (\alpha-1,k\alpha-1,\gamma-k\alpha+1) &\sim (\alpha+1,k\alpha-3,\gamma-k\alpha+1) \\
    &\sim (\alpha+1,\alpha-2,\gamma-\alpha) \sim (1, \alpha-2,\gamma)
\end{align*}
If $a_3<\alpha-1$, assume $\tau\geq 2$, then
\begin{equation*}
(a_{1},a_{2},\gamma-a_{2}) \sim (a_{1}+1,a_{2}-1,\gamma-a_{2})
\sim (a_{1}+1,\alpha-1,\gamma-\alpha)
\sim (1,a_{1}+\alpha-1,\gamma-\alpha)
\end{equation*}
Otherwise, recall the deinition of $e$. Assume $e\leq\alpha-4$ and $a_1\leq e+1$. Then
\begin{align*}
    (a_{1},a_{2},\gamma-a_{2}) &\sim (a_{1}+1,a_{2}-1,\gamma-a_{2}) \sim (a_{1}+1,\alpha-1,\gamma-\alpha) \\
    &\sim (e+3,\alpha-1,\gamma-\alpha-(e+2-a_{1})) \\
    &\sim (1,\alpha+e+1,\gamma-\alpha-(e+2-a_{1})).
\end{align*}
Now the case $e\geq \alpha-3$ was already adressed. And if $a_1\geq e+2$, then
\begin{align*}
    (a_{1},a_{2},\gamma-a_{2}) &\sim (a_{1}+1,a_{2}-1,\gamma-a_{2}) \sim (a_{1}+1,\alpha-1,\gamma-\alpha) \\
    &\sim (a_1-(e+1),\alpha+e+1,\gamma-\alpha) \sim (a_1-(e+1),e+1,\gamma)
\end{align*}

Note that the last two cases of \eqref{exclem} assume that $C$ is Kunz. Thus we have already proved that, otherwise, the ideal of $C'$ is given by
$$
I(C')=\langle X_{a_{s}}X_{b_{s}}-X_{a_{s_{i}}}X_{b_{s_{i}}} \rangle
$$ 
for $s \in \{2, \ldots, \gamma\}$ and $i \in \{1, \ldots, \nu_{s}\}$.

{\bf Step 4:} $(k\alpha+r,\gamma/2,\gamma/2)$ as in case (iv)

\noindent Once again we will assume here $3\in\gap$, otherwise, as said above, we will address the case at the end.  

\noindent Assume $1< r< \alpha-1$. Then
\begin{equation*}
(k\alpha+r,\gamma/2,\gamma/2) \sim (k\alpha+r+1,\gamma/2-1,\gamma/2)
\sim (1,\gamma/2-1,\gamma/2+k\alpha+r)
\end{equation*}
Assume $r=1$. Then
\begin{align*}
(k\alpha+1,\gamma/2,\gamma/2) &\sim (k\alpha+3,\gamma/2-2,\gamma/2)
\sim (k\alpha+3,\gamma/2-\alpha-3,\gamma/2+\alpha+1) \\
&\sim ((k-1)\alpha+3,\gamma/2-\alpha-3,\gamma)
\end{align*}
Assume $r=\alpha-1$. Then
\begin{align*}
    (k\alpha+\alpha-1,\gamma/2,\gamma/2) &= ((k+1)\alpha-1,(m+1)\alpha-1,(m+1)\alpha-1) \\
    &\sim ((k+1)\alpha+1,(m+1)\alpha-3,(m+1)\alpha-1) \\
    &\sim ((k+1)\alpha+1,m\alpha-2,(m+2)\alpha-2) \\
    &\sim (1, m\alpha-2,(m+k+3)\alpha-2)
\end{align*}

{\bf Step 5:} $(\gamma/2,\gamma/2,\gamma/2)$ and $\sss$ is pseudo-symmetric as in case (v). 

\noindent 
For the remainder say an exceptional triple is \emph{irreducible} if it is not equivalent to $(1,a_{s}.b_{s})$, or $(a_{s},b_{s},\gamma)$, or another exceptional triple. It is easily seen that $(\gamma/2,\gamma/2,\gamma/2))$ is irreducible. Thus, if $3 \in \gap$, \eqref{somadireta} yields 
from what was said above, that 
$$
I = \langle  X_{a_{s}}X_{b_{s}}-X_{a_{s_{i}}}X_{b_{s_{i}}}, X_{\gamma/2}^{3}-X_{1}X_{a}X_{b}, X_{\gamma/2}^{3}-X_{a'}X_{b'}X_{\gamma}\rangle.
$$
where $(a,b)$ is the minimal decomposition of $3\gamma/2-1$ and $(a',b')$ is the minimal decomposition of $\gamma/2$.
It follows that $C'$ is cut out by quadrics and cubics.

Now we address the case $3\in \sss$. Note that, if so, $(2,\gamma/2, \gamma/2)$ and $(3k+r, \gamma/2,\gamma/2)$ with $1 \leq k \leq m-1$, $0 < r <3$, $\gamma=2(3m+r)$ are irreducible.  Thus, \eqref{somadireta} yields 
\begin{align*}
    I(C') = &\langle  X_{a_{s}}X_{b_{s}}-X_{a_{s_{i}}}X_{b_{s_{i}}},  X_{\gamma/2}^{3}-X_{1}X_{a}X_{b}, X_{\gamma/2}^{3}-X_{a'}X_{b'}X_{\gamma}, \\
    & X_{2}X_{\gamma/2}^{2}-X_{1}^{2}X_{\gamma}, X_{3k+r}X_{\gamma/2}^{2}-X_{1}X_{c}X_{d}, X_{3k+r}X_{\gamma/2}^{2}-X_{c'}X_{d'}X_{\gamma}\rangle
\end{align*}
where $(c,d)$ is the minimal decomposition of $3k+r+\gamma-1$ and $(c',d')$ is the minimal decomposition of $3k+r$.
\end{proof}


\begin{rem} \label{remebm} As said in the Introduction, the proof above was highly inspired by St\"ohr's \cite{St1}. There are just a few differences we describe next, namely: (a) Here we consider the semigroup of a singularity instead of the Weierstrass semigroup of a simple point; if $C$ is monomial, then the former can be recovered from the latter, where the simple point is the infinity; (b) we opted to put the theory of \emph{exceptional monomials} developed in \cite{St1} within a purely combinatorial form, and reducing to the case of triples, which is what really counts for the sake of the proof; (c) Equation \eqref{somadireta} is crucial for all the argument; in \cite{St1}, it is derived from the projectively normality of the canonical curve, while here it comes from the intrinsic version of Max Noether's theorem proved in \cite{GM} and some dimension counts developed in \eqref{thmdim}; so $C'$ need not to be projectively normal for the proof to work; for instance, note we assumed $C$ is nearly Gorenstein only after \eqref{somadireta} is proved; the reason why is that it is easier to deal with the exceptional triples in this case, but we hope targeting the general problem in a near future; (d) \cite{St1} goes further in the description of the syzygies, as can be seen from the appearance of the coefficients $c_{sir}$ \cite{St1}*{p.~193} whose relations will later describe the moduli of Gorenstein curves with a fixed Weierstrass point; rather, here those coefficients do not appear as our curves are assumed to be monomial; if we had considered those coefficients $c_{sir}$, then we would be describing the moduli of (non-necessarily monomial) rational curves with a prescribed semigroup at a given cusp, which also seems an interesting problem to be address in a forthcoming work; (e) the \emph{pseudo symmetric} semigroups, which characterize Kunz curves, correspond to Oliveira-St\" ohr's \emph{quasi-symmetric} semigroups, addressed by them in \cite{OS}; but in this case, the correspondence is misleading; indeed, in \cite{OS} those curves are Gorenstein though with a non-symmetric Weierstras semigroup, while here Kunz curves are always non-Gorenstein.
\end{rem}
\end{sbs}

\begin{sbs}[Enriques-Babbage's Theorem for Monomial Curves] \label{ebtmc}
Now we characterize trigonal unicuspidal monomial curves and check when they happen to be nearly Gorenstein.

\begin{thm} \label{trigonalchar}
Let $C$ be a unicuspidal monomial trigonal curve with semigroup $\sss$. Then,
\begin{itemize}
\item[(i)] If $\sss = \{ 0, \alpha, \alpha+1, \cdots, \alpha+k, \alpha+k+\ell,\to \}$ for $\alpha\geq 3$, $k\geq 0$, and $\ell\geq 2$, then $C$ is nearly Gorenstein if and only if $C'\ncong\mathbb{P}^1$; 
\item[(ii)] If $\sss = \{0, \alpha, \alpha+2, \cdots, \alpha+2k,\to\}$ for $\alpha\geq 3$, and $k\geq 1$, then $C$ is nearly Gorenstein;
\item[(iii)] If $\alpha=3$ and $\alpha\neq\beta$, then $C$ is nearly Gorenstein if and only if $C$ is Kunz. 
\end{itemize}
\end{thm}
\begin{proof} The characterization of the three possible semigroups was already done in \cite{FGMS}*{Thm.~5.3}. To prove the remaining statement we address those cases.

\noindent {\bf case (i):}

Note that $\kk=\{0, \ldots, \gamma-(\alpha+k+1) \} \cup \{\gamma-\alpha+1, \ldots, \gamma-1\} \cup \{\gamma+1, \to \}$. Then, by \eqref{ngumc}, it is easily seen that $C$ is nearly Gorenstein if and only if $\gamma-(\alpha+k+1)=0$. Indeed, otherwise $1\in\kk$, and hence $\gamma-(\alpha+k)\in\langle \kk\rangle$. But $\gamma-(\alpha+k)\not\in\kk\cup\{\gamma\}$. Now $\gamma-(\alpha+k+1)=0$ iff $\alpha+k=\gamma-1$ iff $\gamma-1\in\sss$ iff $1\not\in\kk$ iff $C'\ncong \mathbb{P}^1$. 

\noindent {\bf case (ii):} 


Now note that $\kk=\{0,2,4,6,\ldots, \gamma-\alpha-1 \} \cup \{\gamma-\alpha+1, \ldots, \gamma-1\} \cup \{\gamma+1, \to \}$. Then, clearly, $\langle\kk\rangle=\kk\cup\{\gamma\}$ and hence $C$ is nearly Gorenstein by \eqref{ngumc}. 

\noindent {\bf case (iii):} 


If $C$ is nearly Gorenstein, it is not Gorenstein, so there exists $a\leq b \in \gap$ such that $a+b=\gamma$.  Say $\gamma\equiv i\ {\rm mod}\ 3$ and $a\equiv j\ {\rm mod}\ 3$. Clearly $i,j\in\{1,2\}$ since $\gamma$ and $a$ are not in $\sss$. Also, $i\neq j$ since $\gamma-a\not\in\sss$. Therefore $2a\equiv\gamma\ {\rm mod}\ 3$ and hence $\gamma-2a\in\sss$, which implies $2a\not\in\kk$. By \eqref{ngumc}, we have $2a=\gamma$, i.e., $a=b=\gamma/2$. As $a,b$ are arbitrary, it follows that $C$ is Kunz. The converse is a general fact. 
\end{proof}
\begin{thm} \label{gororngor}
Let $C$ be a unicuspidal monomial curve of genus $g\geq 3$ whose canonical model $C'$ is linearly normal. Then one, and only one, of the following holds 
\begin{itemize}
    \item[(i)] $C'$ is cut out by quadrics or;
    \item[(ii)] $C$ is trigonal and Gorenstein or; 
    \item[(iii)] $C$ is isomorphic to a plane quintic or; 
    \item[(iv)] $C$ is Kunz. 
\end{itemize}
Moreover, assume $C$ is trigonal and non-Gorenstein. If the $g_{3}^{1}$ is base point free, then $C$ satisfies case (iv). Otherwise, it may be included in cases (i) or (iv). 
\end{thm}
\begin{proof}
Firt assume $C$ is hyperelliptic. Then $C'$ is the rational normal curve of degree $g-1$ in $\mathbb{P}^{g-1}$ by \cite{KM}*{Prp.~2.6}. So $C'$ is cut out by quadrics and (i) holds. But if $C$ is hyperelliptic, then $\gon(C)=2$ so $\gon(C)\neq 3,4$ so (ii) and (iii) does not hold; also, $C$ is Gorenstein so (iv) does not hold as well.

Assume $C$ is non-hyperelliptic Gorenstein and that the semigroup $\sss$ at the singular point is different from $\langle 2, 2g+1 \rangle$, $\langle 3, g+1 \rangle$, $\{0,g,g+1,\ldots,2g-2,2g,\to\}$ and $\langle 4, 5 \rangle$. Then \cite{CS}*{Lem.~2.2} implies the canonical embedding of $C$, which is isomorphic to $C'$, is cut out by quadrics, so (i) holds. Now note that, by \eqref{trigonalchar}, $C$ is trigonal Gorenstein if and only if $\sss=\{0,g,g+1,\ldots,2g-2,2g,\to\}$ or $\sss=\langle 3,g+1 \rangle$. Indeed, those cases correspond to first and third items of \eqref{trigonalchar} if one assumes $\sss$ is symmetric since $C$ is Gorenstein; while the second item of \eqref{trigonalchar} is always a non-symmetric semigroup. So (ii) does not hold. As $\sss\neq\langle 4,5\rangle$, then (iii) does not hold as well. And as $C$ is Gorenstein, then (iv) is precluded.

If $\sss=\langle 2,2g+1\rangle$, then $C$ is hyperelliptic and we repeat the first paragraph of this proof. 

If $\sss=\{0,g,g+1,\ldots,2g-2,2g,\to\}$ or $\sss=\langle 3,g+1 \rangle$, then, as said above, $C$ is trigonal and Gorenstein. So (ii) holds. But (iii) and (iv) clearly do not. To prove (i) is also false, we appeal to \cite{RSt}*{Prp.~3.3}: if $C$ is a trigonal Gorenstein canonical curve then the intersection of all quadrics containing $C$ is a rational normal scroll of dimension $2$, so (i) does not hold.

If $\sss=\langle 4,5\rangle$ then (iii) holds. Also, (i) clearly does not; (ii) does not since $\gon(C)=4\neq 3$; and (iv) is precluded as $C$ is plane and hence Gorenstein.

Finally, if $C$ is non-Gorenstein then $C'$ is lineary normal if and only if $C$ is nearly Gorenstein owing to \cite{KM}*{Thm.~6.5}. Therefore \eqref{EBmonomial} does the trick.

Now assume, $C$ is trigonal non-Gorenstein computed by a base point free pencil. Then the proof of \eqref{trigonalchar} yields $3 \in \sss$ and the $g_3^1$ is $\oo_C\langle 1,t^3\rangle$. Thus \eqref{trigonalchar}.(iii) implies (iv) holds as $C$ is nearly Gorenstein. If the gonality is computed by a pencil with an irremovable base point, then the proof of \eqref{trigonalchar} clearly imply (i) or (iv) may hold. 
\end{proof}

\begin{rem}
\label{remebb}
As said in the Introduction, we conjecture this result may hold in general. Indeed, the argument above for a hyperelliptic curve is general. If the curve is Gorenstein, with gonality at least 4, with a Weierstrass point with symmetric semigroup then one could apply St\"ohr \cite{St1}*{Thm.~2.6}, or Contiero-St\"ohr \cite{CS}*{Thm.~2.5}. To drop this hypothesis, if $k$ is of characteristic $0$, one could consider Schreyer's approach to Petri's analysis in \cite{Sc}*{Thm.~3.1}, as pointed out by Oliveira-St\"or in \cite{OS}*{p.~56, bot}. And if $C$ is non-Gorenstein then \eqref{Kleiman} gets close: non-Kunz nearly Gorenstein curves are cut out by quadrics, so it suffices to prove that Kunz curves are not.
\end{rem}
\end{sbs}

Next we build a family of unicuspidal monomial Kunz curves of arbitrary high Clifford index, then violating Green's Conjecture at level $p=1$, as, already seen, the Koszul cohomology of $\ww$ doesn't vanish for those curves.

\begin{thm} \label{cliffgthm2}
For every $n \in \mathbb{N}$, there exists at least one monomial unicuspidal Kunz curve of Clifford index $n$.
\end{thm}

\begin{proof}
We will prove this result by giving an explicit family of curves that satisfy the hypothesis and then we will calculate their Clifford indexes. For $\alpha \geq 3$, consider the family of curves $(1:t^{\alpha}:t^{\alpha(\alpha-2)+3}:t^{\alpha(\alpha-2)+4}:\cdots:t^{\alpha(\alpha-1)}:t^{\alpha(\alpha-1)+1}) \subset \mathbb{P}^{\alpha}$ whose semigroups have as minimal sets of generators $\langle 3, 7 \rangle$ if $\alpha=3$ and $\langle \alpha, \alpha(\alpha-2)+3, \alpha(\alpha-2)+4,\cdots:\alpha(\alpha-1)-1,\alpha(\alpha-1)+1\rangle$ if $\alpha \geq 4$. These curves are clearly monomial and unicuspidal.  To see this, let us first describe the structure of $\sss$. For short, for any integers $a,b$, we set $[a,b]:=[a,b]\cap\nn$. So we may write
\begin{equation} \label{structureofS}
   [1,\gamma]=\gap_1\cup \sss_1\cup\gap_2\cup\ldots\cup\sss_{2(\alpha-2)}\cup\gap_{2\alpha-3} 
\end{equation}
with
\begin{equation} \label{intgap}
       \gap_{i} = \begin{cases}
      [(i-1)\alpha+1\, ,\, i\alpha-1] & \mbox{if $1 \leq i \leq \alpha-2$} \\
      [(\alpha-2)\alpha+1,(\alpha-2)\alpha+2] & \mbox{if $i = \alpha-1$} \\
      \{(i-\alpha+1)\alpha+(\alpha-2)\alpha+2\} & \mbox{if $\alpha \leq i \leq 2\alpha-3$} 
    \end{cases} 
\end{equation}

and
\begin{equation} \label{ints}
       \sss_{i} = \begin{cases}
      i\alpha & \mbox{if $1 \leq i \leq \alpha-2$} \\
      [(i-1)\alpha+3, i\alpha+1] & \mbox{if $\alpha -1 \leq i \leq 2\alpha-4$}.
    \end{cases}
\end{equation}
First, we will show that its Frobenius number is given by $(2\alpha-4)\alpha+2$. By \eqref{ints}, it does not belong to $\sss$ and also $[(2\alpha-4)\alpha+3, (2\alpha-3)+1] \subset \sss$. We have that $(2\alpha-4)\alpha+2+\alpha = (\alpha-2)\alpha+3+(\alpha(\alpha-1)-1) \in \sss$, so the result follows. 

Also, it can be verified that these semigroups are pseudo-symmetric. To this end we will use \cite{GR}*{Prop.\,4.4}. Clearly, $\gamma=2(\alpha-2)\alpha+2$ is an even number and by construction $\gamma/2=(\alpha-2)\alpha+1 \notin \sss$. So given a gap $\ell \in \gap$, we have to show that $\gamma-\ell \in \sss$. For instance, let $\ell=(i-1)\alpha+j$ for some $i \in [1, \alpha-2]$ and $j \in [1,\alpha-1]$. $\gamma-[(i-1)\alpha+j]=(2\alpha-3-i)\alpha+(2-j)$. So the result follows by considering the possible values of $i$ and $j$. The remaining cases are handled similarly. Hence the previously described curves are Kunz. 

Let 
\begin{equation} \label{monsheaf}
\fff:=\oo_{C}\langle 1, t^{a_{1}}, \ldots, t^{a_{n}} \rangle    
\end{equation}
be a sheaf on $C$ that contributes to the Clifford index. Set $\ff :=v(\fff_P)$. First, suppose that $\fff=\oo\langle 1,t^{a_n} \rangle$ with $a_{n} \in \ff \cap \gap_{m+1}$, so $a_{n}=m\alpha+j$ for some $m \in [0, \alpha-2]$ and $j \in [1,\alpha-1]$. Also, let $a_{n_i}$ be the positive integers such that $a_{n_i} \equiv a_{n} \mod \alpha$ and $a_{n_i} \leq a_{n}$. Note that if $a_{n} \equiv 1 \mod \alpha$ or $a_{n} \equiv 2 \mod \alpha$:
\begin{align} \label{cond1}
\cliff(\fff+\sum_{i=1}^m t^{a_{n_i}}\oo) & = (m-1)(\alpha-2)+(k-1)+(2\alpha-3-(m-1)) \nonumber \\ 
& \leq (m-1)(\alpha-1)+(k-1)+ (2\alpha-3-(m-1)) \nonumber \\
& \leq\cliff(\fff+t^{a_{n}}\oo).
\end{align}
On the other hand, if $a_{i} \equiv j \mod \alpha$ for $j \in [3, \alpha-2]$, we have 
\begin{align} \label{cond2}
\cliff(\fff+\sum_{i=1}^m t^{a_{n_i}}\oo) & = m(\alpha-2)+(k-1)+(2\alpha-4-(m+1)) \nonumber \\ 
& = m(\alpha-1)+(k-1)+(2\alpha-4-(m+1)-m)  \nonumber \\
& =\cliff(\fff+t^{a_{n}}\oo).
\end{align}
So, given a global section $t^{a_{n}}$, the sheaf with the additional global sections $t^{a_{n_i}}$ will have a less than or equal Clifford index to the original sheaf. In order to prove the remaining result, we will consider all possible intervals that $a_{n}$ can belong according to \eqref{ints} and \eqref{intgap}. Also, set ${\rm E}:=\bigcup_{i=1}^n(a_i+\sss)$. Then $\ff=\sss\cup{\rm E}$.

\noindent {\bf case 1:} 
$a_{n} \in \sss_{i}$ for $1 \leq i \leq \alpha-2$.

Let $k:=\ff\cap\gap_{1}$ and $l \in [1, \alpha-2]$. By \cite{FGMS}*{Lem.\,4.1}, we get
{\allowdisplaybreaks
\begin{align*}
\cliff(\fff)&=\#((\gap\setminus{\rm E})\cap[1,a_n])-\#(\sss\cap[1,a_n])+\#(({\rm E}\setminus\sss)\cap[a_n+1,\gamma]) \\
&=\sum_{i=1}^{l}(\#(\gap_{i}\setminus{\rm E})-\#\sss_{i})+\sum_{i=l+1}^{\alpha-2} \#({\rm E}\cap \gap_{i}) \\
& \ \ \ \ \ \ \ \ + \mathds{1}_{\{((1+l\alpha) \cap \ff)\neq\varnothing\}}+\mathds{1}_{\{((2+l\alpha)\cap \ff) \neq \emptyset \}}+(\alpha-2)\mathds{1}_{\{\ff \setminus \sss \neq \emptyset \}} \\
&= l((\alpha-1-k)-1)+k(\alpha-2-l) \\
& \ \ \ \ \ \ \ \ + \mathds{1}_{\{((1+l\alpha) \cap \ff)\neq\varnothing\}}+\mathds{1}_{\{((2+l\alpha)\cap \ff) \neq \emptyset \}}+(\alpha-2)\mathds{1}_{\{\ff \setminus \sss \neq \emptyset \}}
\end{align*}
}
where $\mathbbm{1}_{A}$ denotes the indicator function with respect to an arbitrary set $A$. The last four terms are clearly non-negative. The first term can only be non-negative if and only if $k=\alpha-1$. But in that case, $\ff = \gap$ and hence $h^{1}(\fff)=0$, by \cite{FGMS}*{Lem.\,4.1}. If the last three terms are zero, we have that $\fff$ is locally free. Hence, $k=0$ and the minimum is attained at $l=1$, which implies $\cliff(\fff)=\alpha-2$. If the third and the fourth terms are simultaneously greater than zero, we can use \eqref{cond1}. But this would imply that $\cliff(\fff) \geq \alpha$ based only on the the last terms. Assume now that third term is greater than zero and the fourth term is equal zero. We can apply \eqref{cond1} again and conclude that $\cliff(\fff) \geq \alpha+2$. The same happens in the remaining case. 

\noindent {\bf case 2:} 
$a_{n} \in \sss_{i}$ for $\alpha \leq i \leq 2\alpha-4$.

By the definition of the Clifford index, given a sheaf $\fff$, the condition  $h^{1}(\fff)\geq 2$ must be verified. Hence, $(i-\alpha+1)\alpha+(\alpha-2)\alpha+2 \notin \ff \cap \gap$ for some $i \in [\alpha, 2\alpha-3]$. This imply that $\cliff(\fff) \geq \alpha -1$.

\noindent {\bf case 3:} $a_{n} \in \gap_{i}$ for $1 \leq i \leq \alpha-2$. 

We can use \eqref{cond1} and \eqref{cond2} which imply that we can choose $\fff$ such that $\ff \cap [1, \alpha-1] \neq \varnothing$. Thus, we can conclude that $\cliff(\fff) \geq \alpha-2$.

\noindent {\bf case 4:} $a_{n} \in \gap_{i}$ for $1 \leq i \leq \alpha-1$.

As described in \eqref{intgap}, these sets have only two elements. Note that if $a_{n}=(\alpha-2)\alpha+2$, we have that $a_{n}+i\alpha$ for $i \in [1, \alpha-2]$ coincides with the last line of \eqref{intgap}. By \cite{FGMS}*{Lem.\,4.1} this implies that $h^{1}(\fff)=0$, so a sheaf of that form cannot contribute to the Clifford index. So $a_{n}=(\alpha-2)\alpha+1$ is the only possibility. Using again the fact that $h^{1}(\fff) \geq 2$ must hold and by applying an similar argument to the second case, $[1, \alpha-1] \cap \ff$ must be empty. This implies that $\cliff(\fff) \geq \alpha-1$.  

\noindent {\bf case 5:} $a_{n}=(i - \alpha + 1)\alpha + (\alpha - 2)\alpha + 2 \in \gap_{i}$ for $\alpha \leq i \leq 2\alpha - 3$. 

This case is clearly not possible as these sheaves will not contribute to the Clifford index, since $h^{1}(\fff)=0$ will necessarily hold. 

Thus, all sheaves of the form \eqref{monsheaf} contributing to Clifford index have Clifford index greater or equal than $\alpha-2$. Therefore the Clifford index is computed by $\fff = \oo\langle1,t^{\alpha}\rangle$, which yields $\cliff(C)=\alpha-2$.
\end{proof}

\end{document}